\documentclass[11pt,a4paper]{article}

\usepackage[utf8]{inputenc}
\usepackage[T1]{fontenc}
\usepackage{lmodern}

\usepackage{amsmath, amssymb, amsthm}
\usepackage{mathtools} % Advanced math tools
\usepackage{bm} % Bold math symbols
\usepackage[affil-it]{authblk}
\usepackage[left=2.5cm,right=2.5cm, top=2.54cm, bottom=2.54cm]{geometry} % Page margins
\usepackage{setspace} % Line spacing
\usepackage{parskip} % Add space between paragraphs
\usepackage{enumitem} % Control over lists
\usepackage{cite}
\usepackage{graphicx} % Images
\usepackage{array, booktabs} % Tables
\usepackage{float} % Force figure/table placement
\usepackage{tabularx,booktabs}
\usepackage{hyperref} % Hyperlinks
\hypersetup{
    colorlinks=true,
    linkcolor=blue,
    citecolor=blue,
    urlcolor=blue
}

\newtheorem{theorem}{Theorem}[section]

\newtheorem{corollary}[theorem]{Corollary}
\newtheorem{proposition}[theorem]{Proposition}

\theoremstyle{definition}
\newtheorem{definition}[theorem]{Definition}
\newtheorem{example}[theorem]{Example}

\usepackage{tikz,pgf}
\usepackage[bf,small,tableposition=top]{caption}
\usepackage{subfig}

\usetikzlibrary{arrows}
\definecolor{wine}{rgb}{2,2,251}
\definecolor{myblue}{rgb}{0.2,0.4,2}
\hypersetup{
 colorlinks=true,
 linkcolor=blue, % custom color
 linktoc=all
}
\usepackage{rotating}
\usepackage{multicol}
\title{\bf Restricted  $r$-Stirling numbers of the second kind}
\author{Abdelghani Mehdaoui }
\affil{National Higher School of Mathematics, Po. Box 75,  Mahelma 16093,  Sidi
Abdellah, Algiers, Algeria. \\
USTHB,  RECITS Laboratory, Po. Box 32, El Alia 16111 Bab Ezzouar, Algiers, Algeria \\
mehabdelghani@gmail.com, abdelghani.mehdaoui@nhsm.edu.dz.}
\date{}

\begin{document}

% Title page
\maketitle

% Abstract
\begin{abstract}
	In this paper, we introduce the $(r,d)$-Stirling numbers of the second kind, a new generalization of the $r$-Stirling numbers obtained by imposing a distance restriction: for any two elements $i$ and $j$ within the same subset, we require that $|i-j|>d$. We establish recurrence relations, an explicit formula, ordinary and exponential generating functions, a reduction identity connecting these numbers to the standard $r$-Stirling numbers, and several combinatorial identities.
\end{abstract}

\section{Introduction}

The Stirling numbers of the second kind, denoted by ${n\brace k}$, count the number of ways to partition a set of $n$ elements into $k$ non-empty subsets. These numbers play a fundamental role in combinatorics, arising naturally in problems of enumeration, polynomial interpolation, and the study of set partitions. Various generalizations of these numbers have been introduced in the literature (see, for instance, \cite{belbachir2014,belbachir20141,belkhir2020,benyi2019,broder1984,carlitz1980,comtet1974,howard1980,howard1984,merris2000,mihoubi2012,mihoubi2016}).

Among these generalizations, the $r$-Stirling numbers of the second kind, denoted by ${n\brace k}_r$, were introduced by Broder~\cite{broder1984} and later rediscovered by Merris~\cite{merris2000}. These numbers extend the classical Stirling numbers by requiring that the first $r$ elements lie in distinct subsets. They satisfy the recurrence relation
\begin{equation}
	{n\brace k}_r = {n-1\brace k-1}_r + k{n-1\brace k}_r
\end{equation}
with ${n\brace k}_r = 0$ for $n< r$ and ${n\brace k}_r = \delta_{k,r}$ for $n=r$, where $\delta$ is the Kronecker delta.
They arise in the expansion of $(x+r)^{n-r}$ in terms of falling factorials $(x)_k = x(x-1)\cdots(x-k+1)$:
\begin{align}\label{eq:rst2}
	(x+r)^{n-r} =  \sum_{k=r}^n {n \brace k}_{r} (x)_{k-r},
\end{align}
and they admit the explicit formula
\begin{align}\label{eq:rst1}
	{n \brace k}_{r} = \frac{1}{(k-r)!}\sum_{j=0}^{k-r} (-1)^j \binom{k-r}{j} (k - j)^{n-r}.
\end{align}

For $r=0$, we recover the classical Stirling numbers of the second kind, ${n\brace k}_0={n\brace k}$. For additional properties and identities of the $r$-Stirling numbers, we refer the reader to \cite{broder1984,merris2000}.

More recently, further extensions have been proposed. Mihoubi and Maamra~\cite{mihoubi2012} introduced the $(r_1,\dots,r_p)$-Stirling numbers, and B{\'e}nyi et al.~\cite{benyi2019} studied restricted $r$-Stirling numbers with applications to combinatorial identities. Arn{\'o}czki and Nyul~\cite{arno2024} investigated restricted $r$-Stirling numbers, also known as $r$-Bessel numbers. Belbachir and Djemmada~\cite{belbachir2023} introduced the $(l,r)$-Stirling numbers via a combinatorial approach. On the other hand, Chu and Wei~\cite{wei2008} studied set partitions with a distance restriction $|i-j|>d$ for elements within the same subset.

In the present paper, we combine both the $r$-restriction and the distance restriction to define  the $(r,d)$-Stirling numbers of the second kind. In addition to requiring that the first $r$ elements lie in different subsets, we impose the condition that any two elements $i$ and $j$ within the same subset satisfy $|i-j|> d$. The formal definition is as follows

\begin{definition}Let ${n\brace k}_{r,d}$ denote the number of ways to partition a set of $n$ elements into $k$ non-empty subsets such that the first $r$ elements lie in different subsets and for any two elements $i$ and $j$ within the same subset, $|i-j|> d$ with $i\neq j$.
\end{definition}

If $(r,d)=(r,0)$, we recover the classical $r$-Stirling numbers, ${n\brace k}_{r,0}={n\brace k}_r$, and if $(r,d)=(0,0)$, we recover the classical Stirling numbers, ${n\brace k}_{0,0}={n\brace k}$.

When $r\leq d$, the $(r,d)$-Stirling numbers reduce to the numbers studied by Chu and Wei~\cite{wei2008}, who showed that
\begin{align}\label{wei2008}
	\displaystyle {n \brace k}_{r,d} = {n-d \brace k-d}\; \text{ if } r\leq d.
\end{align}

Therefore, in what follows we only consider the case where $r>d$.

\begin{example}

	Let us compute ${5\brace 3}_{2,1}$,

	\begin{center}
		\(\{\{\overline{1}\}, \{\overline{2},4\}, \{3,5\}\}\); \(\{\{\overline{1}, 3\}, \{\overline{2}, 4\}, \{5\}\}\);
		\(\{\{\overline{1}, 3\}, \{\overline{2}, 5\}, \{4\}\}\); \\ \ \\
		\(\{\{\overline{1}, 4\}, \{\overline{2}, 5\}, \{3\}\}\);
		\(\{\{\overline{1}, 4\}, \{\overline{2}\}, \{3,5\}\}\);
		\(\{\{\overline{1}, 5\}, \{\overline{2}, 4\}, \{3\}\}\);\\ \ \\
		\(\{\{\overline{1}, 3, 5\}, \{\overline{2}\}, \{4\}\}\).
	\end{center}

	So ${5\brace 3}_{2,1}=7$.
\end{example}

\begin{theorem} Let $n\geq k\geq r$, where $d$ and $r$ are non-negative integers. Then ${n\brace k}_{r,d}$ satisfies the following recurrence relation
	\begin{equation}\label{rec:3}
		{n\brace k}_{r,d} = {n-1\brace k-1}_{r,d} + (k-d){n-1\brace k}_{r,d}
	\end{equation}
	with ${ 0 \brace  0 }_{r,d} = 1$, ${ n \brace  k }_{r,d} = 0$ for $k < r$, and ${ r \brace  r }_{r,d} = 1.$
\end{theorem}

\begin{proof}

By considering the placement of the $n$-th element, two cases arise. If the $n$-th element forms a singleton subset, then the remaining $n-1$ elements can be partitioned into $k-1$ subsets in ${n-1\brace k-1}_{r,d}$ ways. Otherwise, the $n$-th element belongs to a subset containing other elements. In this case, we first partition the $n-1$ elements into $k$ subsets in ${n-1\brace k}_{r,d}$ ways, and then insert it into one of the $(k-d)$ admissible subsets. Summing  these two cases yields the desired result.
\end{proof}

Using the recurrence relation~\eqref{rec:3}, we compute the $(r,d)$-Stirling numbers of the second kind for small values of $n$ and $k$. Tables~\ref{table:1} and~\ref{table:2} display the first few rows of ${n\brace k}_{r,d}$ for $r=2,3$ with $d=1$, and for $r=5$ with $d=2,3$, respectively.

\begin{table}[H]
	\centering
	{\small
	\subfloat[${ n \brace  k }_{2,1}.$]{%
	\begin{tabular}{l|llllllll}
		$n/k$ & $2$ & $3$ & $4$ & $5$ & $6$ & $7$ & $8$  \\ \hline
		$2$ & $1$ \\
		$3$ & $1$ & $1$ \\
		$4$ & $1$ & $3$ & $1$ \\
		$5$ & $1$ & $7$ & $6$ & $1$ \\
		$6$ & $1$ & $15$ & $25$ & $10$ & $1$ \\
		$7$ & $1$ & $31$ & $90$ & $65$ & $15$ & $1$ \\
		$8$ & $1$ & $63$ & $301$ & $350$ & $140$ & $21$ & $1$ \\

		\end{tabular}
	}%
	\qquad% --- set horizontal distance between tables here
	\subfloat[${ n \brace  k }_{3,1}.$]{%
	\begin{tabular}{l|lllllll}
		$n/k$ & $3$ & $4$ & $5$ & $6$ & $7$ & $8$& $9$ \\ \hline
		$3$ & $1$ \\
		$4$ & $2$ & $1$ \\
		$5$ & $4$ & $5$ & $1$ \\
		$6$ & $8$ & $19$ & $9$ & $1$ \\
		$7$ & $16$ & $65$ & $55$ & $14$ & $1$ \\
		$8$ & $32$ & $211$ & $285$ & $125$ & $20$ & $1$ \\
		$9$ & $64$ & $665$ & $1351$ & $910$ & $245$ & $27$ & $1$
		\end{tabular}
	}}
	\caption{The first rows of ${ n \brace  k }_{r,d}$ for $r=2,3$ and $d=1$. }\label{table:1}
\end{table}

\begin{table}[H]
	\centering
	{\small
	\subfloat[${ n \brace  k }_{5,2}.$]{%

		\begin{tabular}{l|lllllll}
			$n/k$ & $5$ & $6$ & $7$ & $8$ & $9$ & $10$  \\ \hline
			$5$ & $1$ \\
			$6$ & $3$ & $1$ \\
			$7$ & $9$ & $7$ & $1$ \\
			$8$ & $27$ & $37$ & $12$ & $1$ \\
			$9$ & $81$ & $175$ & $97$ & $18$ & $1$ \\
			$10$ & $243$ & $781$ & $660$ & $205$ & $25$ & $1$
			\end{tabular}
	}%
   \qquad  \qquad% --- set horizontal distance between tables here
	\subfloat[${ n \brace  k }_{5,3}.$]{%

		\begin{tabular}{l|lllllll}
			$n/k$ & $5$ & $6$ & $7$ & $8$ & $9$ & $10$ \\ \hline
			$5$ & $1$ \\
			$6$ & $2$ & $1$ \\
			$7$ & $4$ & $5$ & $1$ \\
			$8$ & $8$ & $19$ & $9$ & $1$ \\
			$9$ & $16$ & $65$ & $55$ & $14$ & $1$ \\
			$10$ & $32$ & $211$ & $285$ & $125$ & $20$ & $1$
			\end{tabular}
	}}
	\caption{The first rows of ${ n \brace  k }_{r,d}$ for $r=5$ and $d=2,3$. }\label{table:2}
\end{table}

\begin{theorem} Let $n\geq k\geq r$, where $d$ and $r$ are non-negative integers. Then ${n\brace k}_{r,d}$ satisfies the following recurrence relation
	\begin{equation}
		{n\brace k}_{r,d} = {n\brace k}_{r-1,d} -(r-d-1){n-1\brace k}_{r-1,d}
	\end{equation}
\end{theorem}

\begin{proof}
We classify the partitions counted by ${n\brace k}_{r-1,d}$ according to the placement of the $r$-th element. Recall that ${n\brace k}_{r-1,d}$ enumerates partitions of $\{1,\dots,n\}$ into $k$ subsets in which the elements $1,\dots,r-1$ lie in distinct subsets and the distance restriction $|i-j|>d$ holds for all $i,j$ in different subsets. If $r$ occupies a subset distinct from those containing $1,\dots,r-1$, then the elements $1,\dots,r$ all lie in distinct subsets. Such partitions are counted by ${n\brace k}_{r,d}$.

Otherwise, $r$ belongs to the same subset as some $j\in\{1,\dots,r-1\}$. The distance restriction $|r-j|>d$ forces $j\le r-d-1$, yielding $(r-d-1)$ admissible choices for $j$. Removing $r$ then produces a valid partition counted by ${n-1\brace k}_{r-1,d}$, contributing $(r-d-1){n-1\brace k}_{r-1,d}$ partitions.

Summing the two cases gives ${n\brace k}_{r-1,d} = {n\brace k}_{r,d} + (r-d-1){n-1\brace k}_{r-1,d}$, which rearranges to the desired result.
\end{proof}

The following special cases follow easily from the recurrence relation
\[
\begin{array}{llll}
		\displaystyle {n \brace k}_{0,0} = {n \brace k},    \quad  \displaystyle {n \brace k}_{0,1} = {n-1 \brace k-1}, \quad \displaystyle {n \brace k}_{r,0} = {n \brace k}_r,\quad \displaystyle  {n \brace n}_{r,d} = 1,
		%&\displaystyle {n \brace k}_{0,d} = {n-d \brace k-d}.
		\\ \ \\   \displaystyle  { n \brace  r }_{r,d} = (r-d)^{n-r}, 		% & \displaystyle {n \brace 0}_{r,d} =?
		\quad \displaystyle{n \brace n-1}_{r,d} = \frac{(n+r-1)(n-r)}{2}-d(n-r).
	\end{array}
\]

\begin{theorem}\label{thm:maintheorem1}
Let $n, k, r, p, d$ be non-negative integers such that $n \ge k \ge r > d$ and $r \ge p > d$. Then
\begin{align}\label{eq:convolution}
    {n \brace k}_{r,d}
    = \sum_{i=0}^{n-r} \binom{n-r}{i}{p+i\brace p}_{p,d}{n-p-i+d\brace k-p+d}_{r-p+d,d}.
\end{align}
\end{theorem}

\begin{proof}
We proceed by induction on $n-r$. For the base case $n=r$ (which implies $k=r$), the left-hand side is $1$, and the right-hand side trivially reduces to the single term $i=0$, giving $\binom{0}{0}{p \brace p}_{p,d}{r-p+d \brace r-p+d}_{r-p+d,d} = 1$.

Assume the identity holds for $n-1$, and let $R(n,k)$ denote the right-hand side of \eqref{eq:convolution}. Applying the binomial coefficients identity $\binom{n-r}{i} = \binom{n-1-r}{i} + \binom{n-1-r}{i-1}$ splits $R(n,k)$ into two sums $S_1$ and $S_2$, where
\begin{align*}
	S_1&= \sum_{i=0}^{n-1-r} \binom{n-1-r}{i} {p+i\brace p}_{p,d}{n-p-i+d\brace k-p+d}_{r-p+d,d}, \\  \text{ and } \\
	S_2&= \sum_{i=1}^{n-r}\binom{n-1-r}{i-1}{p+i\brace p}_{p,d}{n-p-i+d\brace k-p+d}_{r-p+d,d}.
\end{align*}
For $S_2$, shifting the index $i \mapsto i+1$ and applying the recurrence ${p+i+1 \brace p}_{p,d} = (p-d){p+i \brace p}_{p,d}$ yields
\[
S_2 = (p-d) \sum_{i=0}^{n-1-r} \binom{n-1-r}{i} {p+i \brace p}_{p,d} {n-1-p-i+d \brace k-p+d}_{r-p+d,d}.
\]

For $S_1$, applying the recurrence to the second factor gives
\[
{n-p-i+d \brace k-p+d}_{r-p+d,d} = {n-1-p-i+d \brace k-1-p+d}_{r-p+d,d} + (k-p){n-1-p-i+d \brace k-p+d}_{r-p+d,d}.
\]

Substituting this expansion into $S_1$ and combining it with $S_2$, the sums recombine directly into evaluations of $R(n,k)$
\[
R(n,k) = R(n-1, k-1) + (k-p)R(n-1, k) + (p-d)R(n-1, k).
\]

Simplifying the coefficients gives $R(n,k) = R(n-1, k-1) + (k-d)R(n-1, k)$. By the induction hypothesis, this exactly matches the recurrence ${n-1 \brace k-1}_{r,d} + (k-d){n-1 \brace k}_{r,d} = {n \brace k}_{r,d}$, which completes the proof.
\end{proof}

\begin{theorem}\label{thm:thmrp} Let $n\geq k\geq r$, where $d$ and $r$ are non-negative integers. Then

	\begin{align}
		{ n \brace k}_{r,d} & = \sum_{i=0}^{n-r} \binom{n-r}{i}{n-r-i\brace k-r}(r-d)^{i}
	\end{align}

\end{theorem}
\begin{proof}
	By substituting $p=r$ in Theorem~\ref{thm:maintheorem1}, we get
	\begin{align*}
{n \brace k}_{r,d} &= \sum_{i=0}^{n-r} \binom{n-r}{i} {r+i\brace r}_{r,d} {n-r-i+d\brace k-r+d}_{d,d}\\
	&=  \sum_{i=0}^{n-r} \binom{n-r}{i}  {n-r-i\brace k-r}(r-d)^{i}.
	\end{align*}
Using formula \eqref{wei2008} and the special case ${n \brace r}_{r,d} = (r-d)^{n-r}$ we get the desired result.
	%{\bf (verified also  by sagemath)}
\end{proof}

Now, we derive an explicit formula for the $(r,d)$-Stirling numbers of the second kind using the Inclusion-Exclusion principle.
\begin{theorem}
For integers $n \ge k \ge r > d$, the $(r,d)$-Stirling numbers of the second kind satisfy the explicit formula:
\begin{align}\label{eq:explicit}
    {n \brace k}_{r,d} = \frac{1}{(k-r)!} \sum_{j=0}^{k-r} (-1)^{j} \binom{k-r}{j} (k - d - j)^{n-r}.
\end{align}
\end{theorem}

\begin{proof}
We interpret ${n \brace k}_{r,d}$ as the number of ways to distribute the set $\{r+1, \dots, n\}$ into $k$ bins, where the first $r$ bins are distinct (containing the first $r$ elements) and the remaining $k-r$ bins are indistinguishable. First, consider the $k-r$ new bins as distinct. The number of choices for placing an element into  $k$ bins is $k-d$. Thus, the total number of unrestricted mappings of $n-r$ elements is $(k-d)^{n-r}$.  Using the inclusion-exclusion principle, let $S_j$ be the set of mappings where at least $j$ specific new bins are empty. The number of available bins becomes $k-j$, reducing the choices to $(k-j)-d$. Thus,
\[
|S_j| = \binom{k-r}{j} (k-d-j)^{n-r}.
\]
The number of surjective mappings is $\displaystyle \sum_{j=0}^{k-r} (-1)^j |S_j|$. To account for the indistinguishability of the $k-r$ new bins, we divide by $(k-r)!$, yielding the explicit formula.

\end{proof}

\subsection{Generating functions}
In this section, we derive both the ordinary and exponential generating functions for the $(r,d)$-Stirling numbers of the second kind.
\begin{theorem}
    For integers $k \geq r \geq 0$, the ordinary generating function for the $(r,d)$-Stirling numbers of the second kind is given by
    \begin{align}
        \sum_{n\geq k} { n \brace k }_{r,d} x^n = \frac{x^k}{\prod_{j=r}^{k} (1-(j-d)x)}.
    \end{align}
\end{theorem}

\begin{proof}
    Let $S_k(x) = \sum_{n\geq k} { n \brace k }_{r,d} x^n$ denote the generating function for a fixed $k$. Using the recurrence relation \eqref{rec:3} by multiplying by  $x^n$ and summing over $n \geq k$, we obtain
    \begin{align*}
        S_{k}(x) &= x \sum_{n\geq k} { n-1 \brace k-1 }_{r,d} x^{n-1} + (k-d)x \sum_{n\geq k} { n-1 \brace k }_{r,d} x^{n-1} \\
        S_{k}(x) &= x S_{k-1}(x) + (k-d)x S_{k}(x).
    \end{align*}
    Rearranging the terms to solve for $S_k(x)$, we find
    \begin{align*}
        S_{k}(x) (1 - (k-d)x) &= x S_{k-1}(x) \\
        S_{k}(x) &= \frac{x}{1-(k-d)x} S_{k-1}(x).
    \end{align*}
    We now determine the base case $S_r(x)$. Since ${n \brace r}_{r,d} = (r-d)^{n-r}$ for $n \ge r$, we have
    \begin{align*}
        S_r(x) &= \sum_{n=r}^{\infty} (r-d)^{n-r} x^n = x^r \sum_{j=0}^{\infty} ((r-d)x)^j = \frac{x^r}{1-(r-d)x}.
    \end{align*}
    By applying backward substitution from $k$ down to $r$
    \begin{align*}
        S_{k}(x) &= \frac{x}{1-(k-d)x} \cdot \frac{x}{1-(k-1-d)x} \cdots \frac{x}{1-(r+1-d)x} S_r(x) \\
        &= \left( \prod_{j=r+1}^{k} \frac{x}{1-(j-d)x} \right) \frac{x^r}{1-(r-d)x} \\
        &= \frac{x^k}{\prod_{j=r}^{k} (1-(j-d)x)}.
    \end{align*}
    This completes the proof.
\end{proof}

%\subsection{Exponential Generating Function}
\begin{theorem}
	 The exponential generating function for the $(r,d)$-Stirling numbers of the second kind is given by,
	\begin{align}\label{gen:exp}
		\sum_{n\geq 0} {n+r \brace k+r }_{r,d} \frac{x^n}{n!} =
		\frac{1}{k!}(e^x - 1)^{k} e^{(r-d)x} ,
	\end{align}
\end{theorem}

\begin{proof}
Using Theorem \ref{thm:thmrp}, we have
	\[
		\begin{aligned}
			\sum_{n \geq 0} {n+r\brace  k+r}_{r,d} \frac{x^n}{n!} & = \sum_{n \geq 0} \sum_{i=0}^n \binom{n}{i}{i\brace k}(r-d)^{n-i} \frac{x^n}{n!}           \\
			 & = \sum_{i\geq k} {i\brace k}  \sum_{n \geq i} \binom{n}{i} \frac{((r-d)x)^{n}}{n!}         \\
			 & = \sum_{i\geq k} {i\brace k} \frac{x^i}{i!}  \sum_{n \geq i} \frac{((r-d)x)^{n-i}}{(n-i)!} \\
			 & = \sum_{i\geq k} {i\brace k} \frac{x^i}{i!}  e^{(r-d)x}                                    \\
			 & = e^{(r-d)x} \frac{(e^x-1)^k}{k!}.
		\end{aligned}\]
\end{proof}

%\subsection{Double Generating Function}
\begin{theorem}The double generating function of the $(r,d)$-Stirling numbers of the second kind is given by

	\begin{align}
		\sum_{n,k} { n+r \brace  k+r }_{r,d} \frac{x^n}{n!} t^k = e^{t(e^x - 1) + (r-d)x}.
	\end{align}
\end{theorem}

\begin{proof}
	% Your proof here
	The result follows directly from the exponential generating function in the formula \eqref{gen:exp} by summing over all $k$.
\end{proof}

\begin{theorem}
Let $n, m, r, p, d$ be integers such that $n \ge m > p \ge r > d$. The $(r,d)$-Stirling numbers of the second kind satisfy the following convolution identity:
\[
\left\{ \begin{matrix} n \\ m \end{matrix} \right\}_{r,d} = \sum_{k=0}^{n-m} \left\{ \begin{matrix} p+k \\ p \end{matrix} \right\}_{r,d} \left\{ \begin{matrix} n-k \\ m \end{matrix} \right\}_{p+1, d}.
\]
\end{theorem}

\begin{proof}
We utilize the ordinary generating function for the $(r,d)$-Stirling numbers:
\[
\sum_{n=m}^{\infty} \left\{ \begin{matrix} n \\ m \end{matrix} \right\}_{r,d} x^n = \frac{x^m}{\prod_{j=r}^{m} (1 - (j-d)x)}.
\]
Splitting the product in the denominator at $p$ and factoring the numerator as $x^m = x^p \cdot x^{m-p}$, we can rewrite this as the product of two distinct generating functions:
\[
\sum_{n=m}^{\infty} \left\{ \begin{matrix} n \\ m \end{matrix} \right\}_{r,d} x^n = \left( \frac{x^p}{\prod_{j=r}^{p} (1 - (j-d)x)} \right) \left( \frac{x^{m-p}}{\prod_{j=p+1}^{m} (1 - (j-d)x)} \right).
\]
The first factor is exactly the generating function for $\left\{ \begin{matrix} i \\ p \end{matrix} \right\}_{r,d}$, while the second factor is $x^{-p}$ times the generating function for $\left\{ \begin{matrix} j \\ m \end{matrix} \right\}_{p+1,d}$. Extracting the coefficient of $x^n$ by expressing this product as a discrete convolution immediately yields the desired identity.
\end{proof}

\begin{theorem} Let $n\geq k\geq p$, where $d$ and $p$ are non-negative integers with $p> d$. Then

	\begin{align}\label{theorem1}
		{ n+r \brace k+r}_{r,d} = \sum_{i=0}^{n-k} \binom{n}{i}{n+p-i\brace k+p}_{p,d}(r-p)^i
	\end{align}

\end{theorem}
\begin{proof}
Applying Theorem~\ref{thm:maintheorem1} with $p$ replaced by $r-p+d$, we obtain
\begin{align*}
	{n+r \brace k+r}_{r,d} = \sum_{i=0}^{n-k} \binom{n}{i} {(r-p+d)+i \brace r-p+d}_{r-p+d,d} {n+p-i \brace k+p}_{p,d}.
\end{align*}
Since ${(r-p+d)+i \brace r-p+d}_{r-p+d,d} = (r-p+d-d)^i = (r-p)^i$ by the special case ${m \brace r'}_{r',d} = (r'-d)^{m-r'}$, the identity follows.
\end{proof}

\begin{theorem}\label{cor:1}
For all non-negative integers $n$ and $r > d \geq 0$, we have
	\begin{align*}
		(x+r-d)^{n} =  \sum_{k=0}^{n} {n+r \brace k+r}_{r,d} (x)_{k}.
	\end{align*}
\end{theorem}
\begin{proof}
We proceed by induction on $n$. For $n=0$, the left-hand side equals $1$, while the right-hand side reduces to ${r \brace r}_{r,d}(x)_0 = 1$.

Assume the identity holds for $n-1$:
\[
(x+r-d)^{n-1} = \sum_{k=0}^{n-1} {n-1+r \brace k+r}_{r,d} (x)_k.
\]
Multiplying both sides by $(x+r-d)$ and writing $x+r-d = (x-k) + (k+r-d)$, we obtain
\[
(x+r-d)(x)_k = (x)_{k+1} + (k+r-d)(x)_k.
\]
Substituting into the induction hypothesis yields
\begin{align*}
(x+r-d)^{n} &= \sum_{k=0}^{n-1} {n-1+r \brace k+r}_{r,d} \bigl[(x)_{k+1} + (k+r-d)(x)_k\bigr] \\
&= \sum_{k=1}^{n} {n-1+r \brace k-1+r}_{r,d} (x)_k + \sum_{k=0}^{n-1} (k+r-d){n-1+r \brace k+r}_{r,d} (x)_k.
\end{align*}
For $1 \leq k \leq n-1$, the recurrence~\eqref{rec:3} gives
\[
{n-1+r \brace k-1+r}_{r,d} + (k+r-d){n-1+r \brace k+r}_{r,d} = {n+r \brace k+r}_{r,d}.
\]
The remaining boundary terms are: for $k=0$, $(r-d){n-1+r \brace r}_{r,d} = (r-d)^{n} = {n+r \brace r}_{r,d}$; for $k=n$, ${n-1+r \brace n-1+r}_{r,d} = 1 = {n+r \brace n+r}_{r,d}$. Combining all terms completes the induction.
\end{proof}

\subsection{Reduction identity and derived properties}

We establish a fundamental relationship connecting the $(r,d)$-Stirling numbers of the second kind to the classical $r$-Stirling numbers.

\begin{theorem}\label{th:relation}
For integers $n \ge k \ge r > d$, the $(r,d)$-Stirling numbers of the second kind satisfy the reduction identity
    \begin{align}\label{eq:relation}
        {n \brace k}_{r,d} = {n-d \brace k-d}_{r-d}.
    \end{align}
\end{theorem}

\begin{proof}
    This result follows immediately by comparing the explicit formula~\eqref{eq:explicit} with the known closed form of the $r$-Stirling numbers given in~\eqref{eq:rst1}.
\end{proof}

Theorem~\ref{th:relation} establishes a direct correspondence between the $(r,d)$-Stirling numbers and the standard $r$-Stirling numbers via a uniform parameter shift. This reduction identity allows us to transfer established properties from the standard theory to our generalized setting. In the results that follow, unless a proof is explicitly given, the result is obtained by applying the reduction identity~\eqref{eq:relation} to the corresponding identity for standard $r$-Stirling numbers found in~\cite{broder1984,nyul2015,mezo2008,mihoubi2016}.

\begin{corollary}
For integers $n \ge r > d \ge 0$, we have
	\begin{align}
		(x-r+d)^{n-r} = \sum_{k=r}^{n} (-1)^{n-k} {n \brace k}_{r,d} x^{(k-r)},
	\end{align}
where $x^{(m)} = x(x+1)\cdots(x+m-1)$ denotes the rising factorial.
\end{corollary}
\begin{proof}
From Theorem~\ref{cor:1}, replacing $x$ by $-x$ and using the relation $(-x)_k = (-1)^k x^{(k)}$, we obtain
\[
(-x+r-d)^n = \sum_{k=0}^{n} {n+r \brace k+r}_{r,d} (-x)_k = \sum_{k=0}^{n} (-1)^k {n+r \brace k+r}_{r,d}\, x^{(k)}.
\]
Multiplying both sides by $(-1)^n$ yields
\[
(x-r+d)^n = \sum_{k=0}^{n} (-1)^{n-k} {n+r \brace k+r}_{r,d}\, x^{(k)}.
\]
Shifting the index via $n \mapsto n-r$ and $k \mapsto k-r$ gives the stated identity.
\end{proof}

\begin{theorem}
For integers $n \ge k \ge r > d$, we have
    \begin{equation}
        {n+1 \brace k+1}_{r,d} = \sum_{j=k-r}^{n-r} (k - d + 1)^{n-r-j} {j+r \brace k}_{r,d}.
    \end{equation}
\end{theorem}

\begin{proof}
    We interpret ${n+1 \brace k+1}_{r,d}$ combinatorially as the number of ways to partition a set of $n+1$ elements into $k+1$ non-empty subsets. Consider the specific element that initiates the $(k+1)$-th subset, and let its position be $j+r+1$.

    The $j+r$ elements preceding it must form exactly $k$ subsets, which contributes ${j+r \brace k}_{r,d}$. The pivot element itself opens the new subset in $1$ way. The remaining $n - r - j$ elements are then distributed among the $k+1$ available subsets. By definition, inserting an element into $k+1$ subsets contributes  $(k+1) - d = k-d+1$, meaning these final elements contribute a combined factor of $(k-d+1)^{n-r-j}$.

    Summing over all possible positions for this pivot element  yields the identity.
\end{proof}
\begin{corollary}
For non-negative integers $n, k$ and $r > d$, we have
\[
\left\{ \begin{matrix} n+k \\ n \end{matrix} \right\}_{r,d} = \sum_{r-d \le i_1 \le \dots \le i_k \le n-d} i_1i_2\cdots i_k.
\]
\end{corollary}

\begin{corollary}
For integers $n \ge k \ge r > d$, we have
	\begin{align}
		{n \brace k}_{r,d} = \sum_{j=k-r}^{n-r} {n-r \brace j} \binom{j}{k-r}(r-d)_{j-k+r},
	\end{align}
\end{corollary}

\begin{corollary}
For integers $n \ge k \ge r > d$, we have
\begin{align}
		{n \brace k}_{r,d} = \frac{1}{(k-r)!} \sum_{i=0}^{r-d} \binom{r-d}{i} {n-r \brace k-d-i } (k-d-i)!.
\end{align}
\end{corollary}

\begin{proposition}\label{pro:25}
For integers $n \ge k \ge r > d$ and $r \ge s \ge 0$, we have
	\begin{align}
		{n \brace k}_{r,d} = \sum_{j=k-r}^{n-r} \binom{n-r}{j} {j+r-s \brace k-s}_{r-s,d}\, s^{n-r-j}.
	\end{align}
\end{proposition}
\begin{proof}
By Theorem~\ref{thm:thmrp}, we have
\[
{n \brace k}_{r,d} = \sum_{i=0}^{n-r} \binom{n-r}{i} {n-r-i \brace k-r}(r-d)^i.
\]
Writing $(r-d)^i = ((r-s-d) + s)^i$ and expanding via the binomial theorem gives
\[
(r-d)^i = \sum_{\ell=0}^{i} \binom{i}{\ell} (r-s-d)^{i-\ell}\, s^{\ell}.
\]
Substituting and exchanging the order of summation (setting $j = i - \ell$):
\begin{align*}
{n \brace k}_{r,d} &= \sum_{i=0}^{n-r} \binom{n-r}{i} {n-r-i \brace k-r} \sum_{\ell=0}^{i} \binom{i}{\ell} (r-s-d)^{i-\ell}\, s^{\ell} \\
&= \sum_{j=0}^{n-r} \sum_{\ell=0}^{n-r-j} \binom{n-r}{j+\ell} \binom{j+\ell}{\ell} {n-r-j-\ell \brace k-r}(r-s-d)^{j}\, s^{\ell}\\
&= \sum_{j=0}^{n-r} \binom{n-r}{j} s^{n-r-j} \underbrace{\sum_{m=0}^{j} \binom{j}{m} {j-m \brace k-r}(r-s-d)^{m}}_{= {j+r-s \brace k-s}_{r-s,d} \text{ by Theorem~\ref{thm:thmrp}}},
\end{align*}
where the last step uses $\binom{n-r}{j+\ell}\binom{j+\ell}{\ell} = \binom{n-r}{j}\binom{n-r-j}{\ell}$ and the inner sum over $\ell$ collapses via the binomial theorem, while the sum over $j$ is recognized as Theorem~\ref{thm:thmrp} applied to ${j+r-s \brace k-s}_{r-s,d}$.
\end{proof}

\begin{corollary}
For integers $n \ge k \ge r > d$, we have
	\begin{align}
		{n \brace k}_{r,d} = \sum_{j=k-r}^{n-r} \binom{n-r}{j} {j+r-1 \brace k-1}_{r-1,d}.
	\end{align}
\end{corollary}
\begin{proof}
This is the special case $s = 1$ of Proposition~\ref{pro:25}.
\end{proof}

\begin{corollary}
For integers $n \ge k \ge r > d$ and $r > s \ge 0$, we have
	\begin{align}
		{n \brace k}_{r,d} = \sum_{j=k-r}^{n-r} {n-s \brace j+r-s}_{r-s,d} \binom{j}{k-r} (s)_{j-k+r}.
	\end{align}
\end{corollary}

\end{document}